\documentclass[11pt,leqno]{amsart}
\usepackage{verbatim,amssymb,amsfonts,latexsym,amsmath,amsthm,bbm,nicefrac,xspace,enumerate}

\newtheorem{theorem}{Theorem}[section]
\newtheorem{lemma}[theorem]{Lemma}
\newtheorem{corollary}[theorem]{Corollary}
\newtheorem{proposition}[theorem]{Proposition}

\theoremstyle{definition}

\theoremstyle{remark}

\newcommand{\R}{\mathbb{R}}

\newcommand{\B}{\mathcal{B}}
\newcommand{\C}{\mathcal{C}}

\newcommand{\F}{\mathcal{F}}

\renewcommand{\P}{\mathcal{P}}

\newcommand{\U}{\mathcal{U}}

\newcommand{\explicitSet}[1]{\left\lbrace #1 \right\rbrace}
\newcommand{\brackets}[1]{\left\langle #1 \right\rangle}
\newcommand{\set}[2]{\explicitSet{#1 \colon #2}}
\newcommand{\seq}[2]{\brackets{#1 \colon #2}}

\renewcommand{\a}{\alpha}
\renewcommand{\b}{\beta}
\newcommand{\g}{\gamma}

\newcommand{\e}{\varepsilon}
\renewcommand{\k}{\kappa}
\newcommand{\s}{\sigma}

\newcommand{\w}{\omega}
\newcommand{\0}{\emptyset}

\newcommand{\sub}{\subseteq}
\newcommand{\rest}{\!\restriction\!}
\newcommand{\cat}{\!\,^{\frown}}

\newcommand{\homeo}{\approx}

\newcommand{\closure}[1]{\overline{#1}}

\newcommand{\wt}{w}

\newcommand{\cf}{\mathrm{cf}}
\newcommand{\card}[1]{\left\lvert #1 \right\rvert}
\newcommand{\tr}[1]{[\![#1]\!]}

\newcommand{\continuum}{\mathfrak{c}}

\newcommand{\dom}{\mathfrak d}

\newcommand{\cov}[1]{\ensuremath{\bld{cov}(\scr{#1})}}

\newcommand{\ch}{\ensuremath{\mathsf{CH}}\xspace}

\newcommand{\zfc}{\ensuremath{\mathsf{ZFC}}\xspace}

\newcommand{\ma}{\ensuremath{\mathsf{MA}}\xspace}

\newcommand{\ser}{\acute{\mathfrak{n}}}
\renewcommand{\mp}{\mathfrak{par}\hspace{-.1mm}}
\renewcommand{\cov}{\mathfrak{cov}\hspace{-.15mm}}

\begin{document}

\title[Covering versus partitioning with the Cantor space]{Covering versus partitioning \linebreak with the Cantor space}
\author{Will Brian}
\address {
W. R. Brian\\
Department of Mathematics and Statistics\\
University of North Carolina at Charlotte\\
9201 University City Blvd.\\
Charlotte, NC 28223, USA}
\email{wbrian.math@gmail.com}
\urladdr{wrbrian.wordpress.com}

\keywords{partitions, Cantor space, Baire space \vspace{1mm}}


\begin{abstract}
What topological spaces can be partitioned into copies of the Cantor space $2^\omega$? An obvious necessary condition is that a space can be partitioned into copies of $2^\omega$ only if it can be covered with copies of $2^\omega$. We prove three theorems concerning when this necessary condition is also sufficient.

If $X$ is a metrizable space and $|X| \leq \mathfrak{c}^{+\omega}$ (the least limit cardinal above $\mathfrak{c}$), then $X$ can be partitioned into copies of $2^\omega$ if and only if $X$ can be covered with copies of $2^\omega$.
To show this cardinality bound is sharp, we construct a metrizable space of size $\mathfrak{c}^{+(\omega+1)}$ that can be covered with copies of $2^\omega$, but not partitioned into copies of $2^\omega$. 

Similarly, if $X$ is first countable and $|X| \leq \mathfrak{c}$, then $X$ can be partitioned into copies of $2^\omega$ if and only if $X$ can be covered with copies of $2^\omega$.
On the other hand, there is a first countable space of size $\mathfrak{c}^+$ that can be covered with copies of $2^\omega$, but not partitioned into copies of $2^\omega$. 

Finally, we show that, regardless of its cardinality, a completely metrizable space can be partitioned into copies of $2^\omega$ if and only if it can be covered with copies of $2^\omega$ if and only if it has no isolated points.
\end{abstract}

\maketitle


\section{Introduction}

Given two topological spaces $X$ and $Y$, let us say $X$ is \emph{$Y$-partitionable} if there is a partition $\P$ of $X$ with each member of $\P$ homeomorphic to $Y$. 
Similarly, let us say $X$ is \emph{$Y$-coverable} if there is a covering $\C$ of $X$ with each member of $\C$ homeomorphic to $Y$. 

Every $Y$-partitionable space is $Y$-coverable, for any $Y$, because every partition is also a covering. 
In general, however, a $Y$-coverable space need not be $Y$-partitionable.
For example, if $Y$ is the circle $S^1$ and $X$ is the wedge sum of two circles (an 8), then $X$ is $Y$-coverable but not $Y$-partitionable.

In this paper, we focus mainly on $Y = 2^\w$, the Cantor space.
As above, any $2^\w$-partitionable space is necessarily $2^\w$-coverable. We prove three theorems concerning when this necessary condition is also sufficient:

\begin{itemize}
\item[$\circ$] A first countable space of size $\leq\! \continuum$ is $2^\w$-partitionable if and only if it is $2^\w$-coverable.
\item[$\circ$] A metrizable space of size $\leq\! \continuum^{+\w}$ (the first limit cardinal above $\continuum$) is $2^\w$-partitionable if and only if it is $2^\w$-coverable.
\item[$\circ$] A Borel subspace of a completely metrizable space is $2^\w$-partitionable if and only if it is $2^\w$-coverable. In particular, a completely metrizable space is $2^\w$-partitionable if and only if it is $2^\w$-coverable if and only if it has no isolated points.
\end{itemize}
These results are proved in Sections 2, 3, and 4, respectively. For the first two results, we give examples showing that the cardinality bounds are necessary and sharp. In other words, we show:
\begin{itemize}
\item[$\circ$] There is a first countable space of size $\mathfrak{c}^+$ that is $2^\w$-coverable but not $2^\w$-partitionable.
\item[$\circ$] There is a metrizable space of size $\continuum^{+\omega+1}$ that is $2^\w$-coverable, but not $2^\w$-partitionable.
\end{itemize}

The main tool for the results of Sections 3 and 4 is a technical lemma proved in Section 3 that recharacterizes the $2^\w$-partitionability of a metrizable space $X$ in terms of the existence of sufficiently dense packings of Cantor spaces into open subsets of $X$.

Finally, in Section 5 we observe that the three main results stated above remain true when $2^\w$ is replaced by the Baire space $\w^\w$, or by any other zero-dimensional Polish space $Y$ with no isolated points. 

It is worth noting that Bankston and McGovern studied the notions of $2^\w$- and $\w^\w$-partitionability in \cite{Bankston&McGovern}. They showed (among other things): 
\begin{enumerate}
\item \ch implies that if $X$ is a completely metrizable space with $|X| \leq \continuum$, then $X$ is $2^\w$-partitionable if and only if $X$ has no isolated points. 
\item \ma implies that if $X$ is a separable completely metrizable space, then $X$ is $2^\w$-partitionable if and only if $X$ has no isolated points.
\end{enumerate}
(See \cite[Theorem 2.10]{Bankston&McGovern}.)
The results of this paper can be seen as improving on these two theorems, both by removing the superfluous set-theoretic hypotheses, and by weakening the topological hypotheses on $X$.


Throughout the paper, all spaces are assumed to be Hausdorff.
We also assume the reader is comfortable with Brouwer's characterization of the Cantor space: up to homeomorphism, $2^\w$ is the unique nonempty zero-dimensional compact metrizable space with no isolated points.

\section{Partitioning small first countable spaces}

Our first lemma gives two simple recharacterizations of $2^\w$-coverability for first countable spaces. 

\begin{lemma}\label{lem:setup}
For any first countable space $X$, the following are equivalent:
\begin{enumerate}
\item $X$ is $2^\w$-coverable.
\item Every $ x \in X$ is contained in a subspace of $X$ homeomorphic to $2^\w$.
\item Every nonempty open subset of $X$ contains a copy of $2^\w$.
\end{enumerate}
\end{lemma}
\begin{proof}
It is clear that $(1)$ implies $(2)$. Conversely $(2)$ implies $(1)$, because if every $x \in X$ is contained in some $K_x \homeo 2^\w$, then $\set{K_x}{x \in X}$ is a covering of $X$ with copies of $2^\w$. Hence $(1)$ and $(2)$ are equivalent.

To show $(2)$ implies $(3)$, suppose $(2)$ holds. Let $U$ be a nonempty open subset of $X$, and fix $x \in U$. By $(2)$, there is some $K \sub X$ with $x \in K \homeo 2^\w$. Because $U \cap K$ is a nonempty open subset of $K$, and because $K$ has a basis of clopen sets homeomorphic to $2^\w$, there is some $C \sub U \cap K$ with $C \homeo 2^\w$. 

It remains to show $(3)$ implies $(2)$, so suppose $(3)$ holds. Let $x \in X$, and fix a countable decreasing neighborhood basis $\set{U_n}{n \in \w}$ for $x$. For each $n \in \w$ let $K_n$ be a copy of $2^\w$ contained in $U_n$ (which exists by $(2)$).
Using Brouwer's topological characterization of the Cantor space, it is not difficult to show that $\{x\} \cup \bigcup_{n \in \w}K_n \homeo 2^\w$. As $x$ was arbitrary, this proves $(2)$.
\end{proof}

\begin{lemma}\label{lem:break-up}
There is a partition $\P$ of $2^\w$ into copies of $2^\w$ such that every nonempty open $U \sub 2^\w$ contains $\continuum$ members of $\P$. 
\end{lemma}
\begin{proof}
The proof makes use of three basic facts about the Cantor space and the Baire space $\w^\w$:
\begin{enumerate}
\item The Cantor space can be partitioned into $\continuum$ copies of itself (for instance, because $2^\w \homeo 2^\w \times 2^\w$ by Brouwer's theorem).
\item The Baire space $\w^\w$ can be partitioned into copies of the Cantor space (because $\w^\w \homeo \w^\w \times 2^\w$; see \cite[Theorem 7.7]{Kechris}.)
\item If a dense $F_\s$ set is removed from $2^\w$, then what remains is homeomorphic to $\w^\w$ (this follows from \cite[Theorems 3.11 and 7.7]{Kechris}).
\end{enumerate}

To prove the lemma, first recursively choose a collection $\set{K_n}{n \in \w}$ of subsets of $2^\w$, such that each $K_n$ is homeomorphic to $2^\w$, each $K_n$ is nowhere dense in $2^\w$, the $K_n$'s are pairwise disjoint, and each basic open subset of $2^\w$ contains some $K_n$. By fact $(3)$, $2^\w \setminus \bigcup_{n \in \w}K_n \homeo \w^\w$. This implies, by fact $(2)$, that there is a partition $\mathcal Q$ of $2^\w \setminus \bigcup_{n \in \w}K_n$ into copies of $2^\w$. By fact $(1)$, there is for each $n$ a partition $\P_n$ of $K_n$ into $\continuum$ copies of $2^\w$. Then $\P = \mathcal Q \cup \bigcup_{n \in \w}\P_n$ is the desired partition of $2^\w$.
\end{proof}

Note also that $2^\w$ can be partitioned into $\aleph_0$ copies of $2^\w$. (This follows from the fact that $2^\w \homeo 2^\w \times (\w+1)$, via Brouwer's theorem.)
However, in contrast to the lemma above, no partition of $2^\w$ into $\aleph_0$ copies of $2^\w$ can have the property that every open set contains $\aleph_0$ partition members: this is a consequence of the Baire Category Theorem, since every member of such a partition would need to be nowhere dense.
Generally speaking, given some ordinal $\a > 0$, it is independent of $\zfc + \continuum > \aleph_\a$ whether there is a partition of $2^\w$ into $\aleph_\a$ copies of $2^\w$ 
(see \cite{Brian&Miller}).

A \emph{packing} of Cantor spaces into a space $X$ is a collection of pairwise disjoint subspaces of $X$, each homeomorphic to $2^\w$. Equivalently, a packing of Cantor spaces into $X$ is a partition of a subspace of $X$ into copies of $2^\w$.

\begin{theorem}\label{thm:gen}
Suppose $X$ is a first countable space with $|X| \leq \continuum$. Then $X$ is $2^\w$-partitionable if and only if $X$ is $2^\w$-coverable.
\end{theorem}
\begin{proof}
If $|X| < \continuum$, then $X$ contains no copies of $2^\w$ and the theorem is trivially true. 
Let $X$ be a first countable space with $|X| = \continuum$. 
As we have already observed, if $X$ is $2^\w$-partitionable then $X$ is $2^\w$-coverable; so it suffices to prove the reverse implication.

Suppose $X$ is $2^\w$-coverable. By Lemma~\ref{lem:setup}, this means every nonempty open subset of $X$ contains a copy of $2^\w$.
Let $\mathcal Q$ be a maximal packing of Cantor spaces into $X$, i.e., a packing that cannot be extended to a larger packing because $X \setminus \bigcup \mathcal Q$ contains no copies of $2^\w$. Some such packing exists by Zorn's Lemma. 

Because $\mathcal Q$ is maximal, every nonempty open subset of $X$ meets some member of $\mathcal Q$. To see this, suppose instead that $U$ is a nonempty open subset of $X$ with $U \cap \bigcup \mathcal Q = \0$. Then there is some $K \sub U$ with $K \homeo 2^\w$. This contradicts the maximality of $\mathcal Q$, since $\mathcal Q \cup \{K\}$ is a strictly larger packing of Cantor spaces into $X$.

For each $K \in \mathcal Q$, there is by Lemma~\ref{lem:break-up} a partition $\mathcal P_K$ of $K$ into copies of $2^\w$ such that every nonempty open subset of $K$ contains $\continuum$-many members of $\mathcal P_K$. 
Let ${\mathcal Q}' = \bigcup_{K \in \mathcal Q} \mathcal P_K$.
Then ${\mathcal Q}'$ is a packing of Cantor spaces into $X$ with the property that every nonempty open subset of $X$ contains $\continuum$-many members of ${\mathcal Q}'$.

Let $Y = X \setminus \bigcup \mathcal Q' = X \setminus \bigcup \mathcal Q$. Let $|Y| = \k$, note that $\k \leq \continuum$, and fix an enumeration $\set{y_\a}{\a < \k}$ of $Y$.

We now choose via transfinite recursion a Cantor space $K(n,\a) \in \mathcal Q'$ for each pair $(n,\a)$ with $n \in \w$ and $\a < \k$.
At stage $\a$ of the recursion, assume $K(n,\xi)$ is already chosen for each $\xi < \a$ and $n \in \w$. Let $\set{U^\a_n}{n \in \w}$ be a countable neighborhood basis for $y_\a$, and
choose $K(n,\a) \in {\mathcal Q}'$, for each $n \in \w$, so that $K(n,\a) \notin \set{K(n,\xi)}{\xi < \a \text{ and } n \in \w}$, and $K(n,\a) \sub U^\a_n$.
Such a choice is possible, since $\card{\set{K(n,\xi)}{\xi < \a \text{ and } n \in \w}} = |\a| \cdot \aleph_0 < \continuum$, while each $U^\a_n$ contains $\continuum$-many members of ${\mathcal Q}'$.

For each $\a < \k$, let
$\textstyle Z_\a = \{y_\a\} \cup \bigcup_{n \in \w} K(n,\a).$
As in the proof of Lemma~\ref{lem:setup}, it follows from Brouwer's characterization of the Cantor space that $Z_\a \homeo 2^\w$ for every $\a < \k$.
Furthermore, if $\a < \b < \k$ then our choice of the $K(\xi,n)$ ensures that $Z_\a \cap Z_\b = \0$. Therefore
$$\P = \set{Z_\a}{\a < \k} \cup \big( \mathcal Q' \setminus \set{K(n,\a)}{n < \w \text{ and } \a < \k} \big)$$
is a partition of $X$ into copies of $2^\w$.
\end{proof}

As a counterpoint to this theorem, we next construct a first countable space of size $\continuum^+$ that is $2^\w$-coverable but not $2^\w$-partitionable. This shows two things: that the cardinality bound $|X| \leq \continuum$ in the previous theorem is necessary (and sharp), and that the ``metrizable'' assumption in Theorem~\ref{thm:c+omega} in the next section cannot be weakened to ``first countable''.

\begin{theorem}\label{thm:firstcountable}
There is a first countable, hereditarily normal space $X$ with $|X| = \continuum^+$ such that $X$ is $2^\w$-coverable but is not $2^\w$-partitionable.
\end{theorem}
\begin{proof}
Let $E$ denote the subset of the ordinal $\continuum^+$ consisting of all $\a < \continuum^+$ with $\mathrm{cf}(\a) = \w$, and let $F$ denote all the successor ordinals in $\continuum^+$. Roughly, our space $X$ is obtained from $E \cup F$ by expanding each member of $F$ to a small copy of $2^\w$.
More precisely, let
$X = E \cup (F \times 2^\w)$
and define a topology on $X$ as follows:
\begin{itemize}
\item[$\circ$] For each $\g \in F$, $\{\g\} \times 2^\w$ is clopen in $X$ and the natural projection mapping from $\{\g\} \times 2^\w$ to $2^\w$ is a homeomorphism.
\item[$\circ$] For each $\a \in E$ and each $\b \in E \cup F$ with $\b < \a$, let 
$$\qquad \quad (\b,\a]_X = \big( (\b,\a] \cap E \big) \cup \big( (\b,\a] \cap F \big) \times 2^\w.$$
The sets of this form constitute a local basis for $\a$.
\end{itemize}

Clearly every point in $F \times 2^\w$ has a countable neighborhood basis in $X$. If $\a \in E$, then there is an increasing sequence $\seq{\b_n}{n \in \w}$ of ordinals in $E \cup F$ with $\a = \sup_{n \in \w} \b_n$. Then $\set{(\b_n,\a]_X}{n \in \w}$ is a countable neighborhood basis for $\a$ in $X$. Hence $X$ is first countable.

Let $Y$ be the totally ordered set $\continuum^+ \times 2^\w$ with the lexicographic order (where we identify $2^\w$ with a subset of $\R$, say the middle-thirds Cantor set). We may identify $X$ with a subset of $Y$, namely $E \times \{0\} \cup F \times 2^\w$, where $0$ denotes the order-minimal member of $2^\w$. Putting the order topology on $Y$, it is straightforward to check that the topology on $X$ is the one it inherits as a subspace of $Y$. Therefore $X$ is a subspace of a LOTS (a Linearly Ordered Topological Space). Every LOTS is $T_5$ (i.e., hereditarily normal), and even more (e.g., every LOTS is hereditarily collectionwise normal \cite{Steen}). Hence $X$ is hereditarily (collectionwise) normal.

It is easy to see $|X| = \continuum^+$. It is also easy to see that $X$ satisfies condition $(3)$ from Lemma~\ref{lem:setup}, which implies that $X$ is $2^\w$-coverable.

It remains to show that $X$ fails to be $2^\w$-partitionable.
Aiming for a contradiction, suppose $\P$ is a partition of $X$ into copies of $2^\w$. 
For each $\a \in E$, let $K_\a$ denote the member of $\P$ containing $\a$.

Note that the topology $E$ inherits from $X$ is the same as the topology it inherits from $\continuum^+$ (when $\continuum^+$ is given its usual order topology). In this topology $E$ is scattered, and in particular contains no copies of $2^\w$. 
Hence for each $\a \in E$, we have $K_\a \not\sub E$, i.e., $K_\a \cap (\{\b\} \times 2^\w) \neq \0$ for some $\b \in F$. Furthermore, there is some $\b \in F$ with $\b < \a$ such that $K_\a \cap (\{\b\} \times 2^\w) \neq \0$, because otherwise $K_\a \cap (0,\a]_X$ would be both a subset of $E$ (which contains no copies of $2^\w$) and a nonempty open subset of $K_\a$ (which does contain a copy of $2^\w$).

For each $\a \in E$, let $f(\a)$ denote the least $\b \in F$ with $K_\a \cap (\{\b\} \times 2^\w) \neq \0$. 
By the previous paragraph, $f(\a) < \a$ for all $\a \in E$; in other words, $f$ is a regressive function on $E$.
Also, $E$ is a stationary subset of $\continuum^+$. 
By Fodor's Lemma, there is a stationary $S \sub E$ and some particular $\b \in F$ such that $f(\a) = \b$ for all $\a \in S$.
Now $|S| = \continuum^+$, and this implies $|\set{K_\a}{\a \in S}| = \continuum^+$ by the pigeonhole principle, since each $K_\a$ has size $\continuum$. Hence there are $\continuum^+$ distinct members of $\mathcal P$ that meet $\{\b\} \times 2^\w$. This is absurd, because the members of $\P$ are disjoint and $\card{\{\b\} \times 2^\w} = \continuum$.
\end{proof}

\section{Partitioning small metrizable spaces}




In this section and the next we turn our attention to progressively more restrictive sub-classes of the first countable spaces: metrizable spaces in this section, and (Borel subsets of) completely metrizable spaces in the next. The following recharacterization of $2^\w$-partitionable metrizable spaces is the main ingredient for the main results in both sections.

\begin{theorem}\label{thm:TFAE}
For any metrizable space $X$, the following are equivalent:
\begin{enumerate}
\item $X$ is $2^\w$-partitionable.
\item There is a partition $\P$ of $X$ into copies of $2^\w$ such that  every open $U \sub X$ contains $|U|$ members of $\P$. 
\item There is a packing $\mathcal Q$ of Cantor spaces into $X$ such that every open $U \sub X$ contains $|U|$ members of $\mathcal Q$. 
\item There is a packing $\mathcal Q$ of Cantor spaces into $X$ such that every non-empty open $U \sub X$ meets some member of $\mathcal Q$, and if $|U| > \continuum$ then $U$ meets $|U|$ members of $\mathcal Q$.
\item For every open $U \sub X$, there is a packing of Cantor spaces into $U$ with cardinality $|U|$.
\end{enumerate}
\end{theorem}

\noindent To prove the main results of this section and the next, we will really only use the equivalence of $(1)$ and $(5)$. 
But our proof that $(1)$ and $(5)$ are equivalent essentially involves proving their equivalence to statements $(2)$ through $(4)$ anyway, so we have chosen to make these waypoints explicit.

\begin{proof}[Proof of Theorem~\ref{thm:TFAE}]
It is clear that $(2) \Rightarrow (3) \Rightarrow (5)$. To prove the theorem, we show that $(5) \Rightarrow (4) \Rightarrow (3) \Rightarrow (1) \Rightarrow (2)$. 

We show first that $(5) \Rightarrow (4)$. Suppose $(5)$ holds.
Let $\mathcal Q$ be a maximal packing of Cantor spaces into $X$, i.e., a packing that cannot be extended to a larger packing because $X \setminus \bigcup \mathcal Q$ contains no copies of $2^\w$. Some such packing exists by Zorn's Lemma. 
We claim this $\mathcal Q$ satisfies $(4)$. 

Let $U$ be a nonempty open subset of $X$. By $(5)$, $U$ contains a copy $K$ of $2^\w$. Because $\mathcal Q$ is maximal, some member of $\mathcal Q$ meets $K$, since otherwise $\mathcal Q \cup \{K\}$ would be a packing of Cantor spaces into $X$ strictly larger than $\mathcal Q$. Hence $U$ meets some member of $\mathcal Q$. 

Now suppose in addition that $|U| = \k > \continuum$.
We claim that $U$ has nonempty intersection with $\k$ members of $\mathcal Q$.
Aiming for a contradiction, suppose instead that $\card{\set{K \in \mathcal Q}{U \cap K \neq \0}} = \mu < \k$.
By $(5)$, there is a packing $\mathcal R$ of Cantor spaces into $U$ with $|\mathcal R| = \k$.
Observe that 
$$\textstyle \card{U \cap \bigcup \mathcal Q} \,\leq\, \card{2^\w} \cdot \card{\set{K \in \mathcal Q}{U \cap K \neq \0}} \,=\, \continuum \cdot \mu \,<\, \k \,=\, |\mathcal R|.$$
It follows that not every member of $\mathcal R$ contains a point of $\bigcup \mathcal Q$; i.e., there is some $K \in \mathcal R$ with $K \cap \bigcup \mathcal Q = \0$. This contradicts the maximality of $\mathcal Q$, as then $\mathcal Q \cup \{K\}$ is a packing of Cantor spaces into $X$ properly extending $\mathcal Q$.
Therefore 
$\card{\set{K \in \mathcal Q}{U \cap K \neq \0}} = \k$, as claimed.

To see that $(4) \Rightarrow (3)$, let $\mathcal Q$ be a packing of Cantor spaces into $X$ having the property described in $(4)$. For each $K \in \mathcal Q$, let $\P_K$ be a partition of $K$ into $\continuum$ pieces such that every nonempty open $V \sub C$ contains $\continuum$ members of $\P_K$. (Such a partition exists by Lemma~\ref{lem:break-up}.) Let ${\mathcal Q'} = \bigcup_{K \in \mathcal Q} \P_K$. Clearly ${\mathcal Q'}$ is a packing of Cantor spaces into $X$, and we claim that it has the property described in $(3)$. To see this, let $U \sub X$ be open. 
Note that if $U \neq \0$, then $|U| \geq \continuum$, because $(4)$ implies that $U$ contains an open subset of a copy of $2^\w$.
If $|U| = \continuum$, then there is some $K \in \mathcal Q$ with $U \cap K \neq \0$. This implies $U \cap K$ contains $\continuum$ members of $\P_K$, hence $U \cap K$ contains $\continuum = |U|$ members of ${\mathcal Q'}$.
If $|U| > \continuum$, then $U$ meets $|U|$ members of $\mathcal Q$; and for each such $K \in \mathcal Q$, there are $\continuum$-many members of $\P_K \sub {\mathcal Q'}$ contained in $U \cap K$. Hence $U$ contains $\continuum \cdot |U| = |U|$ members of ${\mathcal Q'}$. 

The argument that $(1) \Rightarrow (2)$ is similar. Let $\P$ be a partition of $X$ into copies of $2^\w$. As in the previous paragraph, obtain a new partition $\P'$ of $X$ by sub-dividing each member of $\P$ into $\continuum$-many copies of $2^\w$, using Lemma~\ref{lem:break-up}. Clearly $\P'$ is a partition of $X$ into copies of $2^\w$, and we claim that it has the property described in $(2)$. 
To see this, let $U \sub X$ be open. Note that $U \neq \0$ implies $|U| \geq \continuum$, because $U$ contains a copy of $2^\w$ by Lemma~\ref{lem:setup}.
If $|U| = \continuum$, then $U$ must have nonempty intersection with at least one member of $\P$, and it follows that $U$ contains $\continuum = |U|$ members of $\P'$.
If $|U| > \continuum$, then $U$ must have nonempty intersection with $|U|$ members of $\P$ (because $\P$ is a partition of $U$ into sets of size $\continuum < |U|$), so $U$ contains $|U|$ members of $\P'$. 

Finally, let us show $(3) \Rightarrow (1)$. 
Let $\mathcal Q$ be a packing of Cantor spaces into $X$ such that every open $U \sub X$ contains $|U|$ members of $\mathcal Q$.
Let $Y = X \setminus \bigcup \mathcal Q$.

Briefly, we describe in this paragraph the strategy of the proof. (This paragraph may be skipped by the reader who wants only a formal proof.) 
Reviewing the proof of Theorem~\ref{thm:gen}, one may describe the main idea as first choosing for each $y \in Y$ some sequence $\seq{K^y_n}{n \in \w}$ of Cantor spaces converging to $y$, then ``borrowing'' these from $\mathcal Q$ (rather, a modification $\mathcal Q'$ of $\mathcal Q$) to cover each point of $Y$ with a copy of $2^\w$, namely $\{y\} \cup \bigcup_{n \in \w}K^y_n$. 
The strategy here is essentially the same, but with the added difficulty that the $K^y_n$ cannot be chosen one at a time via a simple transfinite recursion. Instead we will need to leverage the paracompactness of $X$ to make all of our choices simultaneously (and with a more complicated indexing).

For each $C \in \mathcal Q$, let $\set{K_{n,k}^C}{n,k \in \w}$ be a partition of $C$ into countably many copies of $2^\w$ (indexed by $\w \times \w$).
Let $\mathcal Q' = \bigcup_{C \in \mathcal Q} \{ K_{n,k}^C \,:\, n,k \in \w \}$.

For each $n \in \w$, there is an open cover $\U_n$ of $X$ such that
\begin{enumerate}
\item the diameter of each $U \in \U_n$ is at most $\frac{1}{n+1}$, and
\item each $x \in X$ has a neighborhood that meets only finitely many members of $\U_n$.
\end{enumerate}
Recall that any open cover satisfying $(2)$ is called \emph{locally finite}; so the existence of some such $\U_n$ is equivalent to the statement that the cover of $X$ by open sets of diameter $\leq \frac{1}{n+1}$ has a locally finite refinement. This is true because every metric space is paracompact \cite{Stone}. Note that $(2)$ implies:
\begin{enumerate}
\item[$(2')$] every compact $C \sub X$ meets only finitely many members of $\U_n$.
\end{enumerate}
To see this, note that, by $(2)$, $C$ can be covered with open sets each of which meets finitely many members of $\U_n$; by compactness, $C$ can be covered by finitely many such sets.

We claim there is an injection $K$ from $\set{(n,U,y)}{y \in Y \cap U \text{ and } U \in \U_n}$ into $\mathcal Q'$ such that $K(n,U,y) \sub U$ for all triples $(n,U,y)$. In other words, we can assign some $K(n,U,y) \in \mathcal Q'$ to each triple $(n,U,y)$ with $y \in Y \cap U$ and $U \in \U_n$ such that different triples index different packing members, and $K(n,U,y) \sub U$. 

For each $n \in \w$ and $U \in \U_n$, let $f_n^U$ be an injective function from $Y \cap U$ into $\set{C \in \mathcal Q}{C \subseteq U}$.
Such an injection exists because $\mathcal Q$ is a witness to $(3)$, which means $|Y \cap U| \leq |U| = \card{\set{C \in \mathcal Q}{C \subseteq U}}$.
For each $n \in \w$ and $C \in \mathcal Q$, there are only finitely many $U \in \U_n$ with $C \sub U$ (by property $(2')$ above); 
fix an enumeration $U^{n,C}_0,\dots,U^{n,C}_\ell$ of these finitely many sets. 
For each $n \in \w$, each $U \in \U_n$, and each $y \in Y \cap U$, define
$K(n,U,y) = K_{n,k}^C$ where $k$ is the (unique) index with $U = U^{n,C}_k$, and $C = f^U_n(y)$.

It follows immediately from the definition of $K$ that $K(n,U,y) \sub U$ whenever $y \in U \in \U_n$.
Let us check that $K$ is injective. 
Because the members of $\mathcal Q$ are pairwise disjoint, $K_{n,k}^C = K_{m,\ell}^{C'}$ only if $C = C'$; and because all the $K_{n,k}^C$ are pairwise disjoint for a given $C \in \mathcal Q$, $K_{n,k}^C = K_{m,\ell}^{C'}$ only if $C = C'$, $n=m$, and $k=\ell$.
It follows that if $m \neq n$, then $K(m,\cdot,\cdot) \neq K(n,\cdot,\cdot)$. 
Now suppose that $U,V \in \U_n$ with $U \neq V$. If $f^U_n(y) = f^V_n(y') = C$ for some $y,y' \in Y$, then $U = U^{n,C}_k$ and $V = U^{n,C}_\ell$ for some $k \neq \ell$, and this implies $K(n,U,y) = K_{n,k}^C \neq K_{n,\ell}^C = K(n,V,y')$. If, on the other hand, $f^U_n(y) \neq f^V_n(y')$, then $K(n,U,y) \neq K(n,V,y')$ because $K(n,U,y) \sub f^U_n(y)$ and $K(n,V,y') \sub f^V_n(y')$. Either way, $U \neq V$ implies $K(n,U,\cdot) \neq K(n,V,\cdot)$. Finally, if $y \neq y'$ then $K(n,U,y) \neq K(n,U,y')$ because $f^U_n$ is injective, and $K(n,U,y) \sub f^U_n(y)$ while $K(n,U,y') \sub f^U_n(y')$.

For each $y \in Y$, let
$$\textstyle Z_y = \{y\} \cup \bigcup \set{K(n,U,y)}{n \in \w \text{ and } y \in U \in \U_n}.$$
Given $y \in Y$, note that $K(n,U,y) \sub U \sub B_{\frac{1}{n+1}}(y)$ whenever $y \in U \in \U_n$, by part $(1)$ of our choice of the $\U_n$.
Therefore each open ball around $y$ contains all but finitely many of the $K(n,U,y)$, each of which is a copy of $2^\w$. 
It follows from this, via Brouwer's characterization of the Cantor space, that $Z_y \homeo 2^\w$ for each $y \in Y$.

Because $K$ is injective, if $y,y' \in Y$ and $y \neq y'$, then $Z_y \cap Z_{y'} = \0$. Thus
$$\P = \set{Z_y}{y \in Y} \cup \set{C \in \mathcal Q'}{C \notin \mathrm{Image}(K)}$$
is a partition of $X$ into copies of $2^\w$.
\end{proof}

Given a topological space $X$, recall that a \emph{cellular family} in $X$ is a collection $\mathcal S$ of pairwise disjoint, nonempty open subsets of $X$. The \emph{cellularity} of $X$, denoted $c(X)$, is defined as
$$c(X) = \sup \set{\card{\mathcal S}}{\mathcal S \text{ is a cellular family in }X}.$$ 
By a theorem of Erd\H{o}s and Tarski \cite{Erdos&Tarski}, if $c(X)$ is not a regular limit cardinal, then there is a cellular family in $X$ of cardinality $c(X)$ (i.e., the supremum in the definition of $c(X)$ is attained). 
It is also proved in \cite{Erdos&Tarski} that if $X$ is metrizable, then the supremum in the definition of $c(X)$ is attained regardless of whether $c(X)$ is a regular limit cardinal or not.

\begin{lemma}\label{lem:sizes}
If $X$ is a metrizable space, then $c(X) \leq |X| \leq c(X)^{\aleph_0}$.
\end{lemma}
\begin{proof}
The inequality $c(X) \leq |X|$ holds in any space, because if $\mathcal S$ is a cellular family in $X$, then any choice function on $\mathcal S$ gives an injection $\mathcal S \to X$, so that $|\mathcal S| \leq |X|$.
If $X$ is a metrizable space, then $X$ has a basis $\B$ of size $c(X)$ (see \cite[Theorem 4.1.15]{Engelking}). Furthermore, every point of $X$ can be determined uniquely by a countable subset of $\B$ (because every point in a metrizable space is $G_\delta$). It follows that $|X| \leq |\B|^{\aleph_0} = c(X)^{\aleph_0}$.
\end{proof}

\begin{lemma}\label{lem:sizes2}
If $X$ is metrizable and $\continuum < |X| \leq \continuum^{+\w}$, then $c(X) = |X|$.
\end{lemma}
\begin{proof}
If $n < \w$, then $(\continuum^{+n})^{\aleph_0} = \continuum^{+n}$. (This well known fact of cardinal arithmetic is easily proved by induction on $n$. For $n = 0$ we have $\continuum^{\aleph_0} = \continuum$. For the inductive step, observe that, because $\continuum^{+n}$ has uncountable cofinality, $(\continuum^{+n})^{\aleph_0} = \sum_{\a < \continuum^{+n}} \a^{\aleph_0} \leq \continuum^{+n} \cdot (\continuum^{+(n-1)})^{\aleph_0} = \continuum^{+n} \cdot \continuum^{+n-1} = \continuum^{+n}$.) 

Applying the previous lemma: if $c(X) \leq \continuum$, then $|X| \leq c(X)^{\aleph_0} \leq \continuum^{\aleph_0} = \continuum$, and if $c(X) > \continuum^{+\w}$ then $|X| \geq c(X) > \continuum^{+\w}$. Therefore, our assumption that $\continuum < |X| \leq \continuum^{+\w}$ implies that $\continuum < c(X) \leq \continuum^{+\w}$.
If $c(X) = \continuum^{+n}$ for some $n < \w$, then $c(X)^{\aleph_0} = c(X)$ by the previous paragraph, and the previous lemma gives $|X| = c(X)$.
If $c(X) = \continuum^{+\w}$, then $\continuum^{+\w} = c(X) \leq |X| \leq \continuum^{+\w}$, and again $|X| = c(X)$. 
\end{proof}

\begin{lemma}\label{lem:sizes3}
If $X$ is metrizable and $2^\w$-coverable, and if $|X| \leq \continuum^{+\w}$, then there is a packing of Cantor spaces into $X$ with size $|X|$.
\end{lemma}
\begin{proof}
If $X = \0$ then the lemma is trivially true. If $X \neq \0$ and $X$ is $2^\w$-coverable, then $|X| \geq |2^\w| = \continuum$. So let us assume $\continuum \leq |X| \leq \continuum^{+\w}$.

If $|X| = \continuum$ and $X$ is $2^\w$-coverable, then $X$ contains a copy $K$ of $2^\w$, and there is a partition of $K$ into $\continuum$ copies of $2^\w$ by Lemma~\ref{lem:break-up}. This partition of $K$ is a packing Cantor spaces into $X$ with size $|X|$.

If $\continuum < |X| \leq \continuum^{+\w}$, then $c(X) = |X|$ by the previous lemma. By the aforementioned results of Erd\H{o}s and Tarski, there is a cellular family $\mathcal S$ in $X$ with $|\mathcal S| = c(X) = |X|$. By Lemma~\ref{lem:setup}, each $U \in \mathcal S$ contains a copy $K_U$ of $2^\w$. Then $\set{K_U}{U \in \mathcal S}$ is a packing of Cantor spaces into $X$ of size $|X|$.
\end{proof}

With these lemmas in place, the main result of this section follows easily:

\begin{theorem}\label{thm:c+omega}
Suppose $X$ is a metrizable space with $|X| \leq \continuum^{+\w}$. Then $X$ is $2^\w$-partitionable if and only if $X$ is $2^\w$-coverable.
\end{theorem}
\begin{proof}
Let $X$ be a nonempty metrizable space (the theorem is trivially true for $X = \0$), and suppose $|X| \leq \continuum^{+\w}$. 
Every $2^\w$-partitionable space is $2^\w$-coverable, so we must show that if $X$ is $2^\w$-coverable then $X$ is also $2^\w$-partitionable. So suppose $X$ is $2^\w$-coverable.

If $U$ is a nonempty open subset of $X$, then the previous lemma implies there is a packing of Cantor spaces into $U$ with size $|U|$. By Theorem~\ref{thm:TFAE}, this implies $X$ is $2^\w$-partitionable.
\end{proof}

To end this section, we now give an example showing that the cardinality bound $|X| \leq \continuum^{+\w}$ in Theorem~\ref{thm:c+omega} is necessary and sharp.

Suppose $F_0, F_1, F_2, \dots$ are ordered sets. If $f,g \in \prod_{n \in \w}F_n$, then we write $f <^* g$ to mean that $f(n) < g(n)$ for all but finitely many $n \in \w$.

\begin{lemma}\label{lem:noCantor}
Let $F_0, F_1, F_2, \dots$ be totally ordered sets, each carrying the discrete topology. A subset of $\prod_{n \in \w}F_n$ cannot be both homeomorphic to $2^\w$ and well ordered by $<^*$.
\end{lemma}
\begin{proof}
Fix $C \sub \prod_{n \in \w}F_n$ with $C \homeo 2^\w$.
For each $n \in \w$, let $P_n$ denote the projection of $C$ onto $F_n$, i.e., $P_n = \set{p \in F_n}{x(n) = p \text{ for some } x \in C}$. Then $\set{\pi_n^{-1}(p)}{p \in P_n}$ is a partition of $C$ into nonempty clopen sets. Because $C$ is compact, $P_n$ is finite. Replacing $F_n$ with $P_n$ if necessary, we may (and do) assume that each $F_n$ is finite.

The relation $<^*$ is a Borel relation on $\prod_{n \in \w}F_n$, and (by the assumption in the previous sentence) $\prod_{n \in \w}F_n \homeo 2^\w$. But it is known that no Borel relation on $2^\w$ can be a well ordering: therefore $<^*$ is not a well ordering of $\prod_{n \in \w}F_n$, nor is its restriction to $C$ (which is a Borel relation on $C$) a well ordering of $C$. (For the fact that no Borel relation on $2^\w$ is a well ordering, see \cite[Theorem 8.48]{Kechris}. An interesting generalization, proved by Reclaw in \cite{Reclaw}, is that if $R$ is any Borel relation on $2^\w$, then no subset of $2^\w$ is ordered by $R$ in type $\omega_1$.)
\end{proof}

\begin{lemma}\label{lem:scale}
Suppose $D$ is a discrete space and $|D|$ has countable cofinality. Then the metrizable space $D^\w$ has a subspace $Y$ with $|Y| = |D|^+$ such that $Y$ contains no copies of $2^\w$.
\end{lemma}
\begin{proof}
For this proof, all cardinals are considered to carry the discrete topology (as well as their usual order).

Let $\k$ be a cardinal with $\cf(\k) = \w$. We claim there is some $Y \sub \k^\w$ such that $|Y| = \k^+$, and $Y$ is well ordered by the relation $<^*$ defined above. This suffices to prove the lemma, because Lemma~\ref{lem:noCantor} implies that any such $Y$ contains no copies of the Cantor set.

This ``claim'' is a fairly well known fact, but for the sake of completeness, we outline a (short) proof here. 

If $\k = \w$, we may take $Y$ be to any $\w_1$-sequence of functions in $\w^\w$ increasing with respect to $<^*$. 

Now suppose $\k$ is uncountable.
Fix an increasing sequence $\seq{\mu_n}{n \in \w}$ of regular infinite cardinals with limit $\k$.


To prove the claim in this case, we show there is a $\k^+$-sized subset $\set{f_\a}{\a < \k^+}$ of $\prod_{n \in \w} \mu_n \sub \k^\w$ that is well ordered by $<^*$.

It is a routine matter to choose the $f_\a$ by recursion, provided that we know: if $\F \sub \prod_{n \in \w} \mu_n$ and $|\F| \leq \k$, then there is some $g \in \prod_{n \in \w} \mu_n$ such that $f <^* g$ for all $f \in \F$. 
So suppose $\F$ is some such family of functions. 
Write $\F$ as an increasing union $\bigcup_{n \in \w}\F_n$, where $|\F_n| \leq \mu_n$ for all $n$. 
Then define $g(0) = 0$ and $g(n) = \sup \set{f(n)+1}{f \in \F_{n-1}}$ for all $n > 0$. Note that $\sup \set{f(n)+1}{f \in \F_{n-1}} \in \mu_n$ because $\mu_n$ is regular and $|\F_{n-1}| \leq \mu_{n-1} < \mu_n$; thus $g$ is well defined. This $g$ clearly works.
\end{proof}

\begin{theorem}\label{thm:counterexample}
There is a metrizable space $X$ with $|X| = \continuum^{+(\w+1)}$ such that $X$ is $2^\w$-coverable but not $2^\w$-partitionable.
\end{theorem}

\begin{proof}
Let $D$ be the discrete space of cardinality $\continuum^{+\w}$. Using the lemma above, fix a subspace $Y$ of $D^\w$ with cardinality $\continuum^{+(\w+1)}$ such that $Y$ contains no copies of $2^\w$. 

Notice that the standard basis $\B$ for $D^\w$ has cardinality $|\B| = |D| = \continuum^{+\w}$. For each basic open $U \sub D^\w$, fix a copy $C_U$ of $2^\w$ such that $C_U \sub U$. Let $$\textstyle X = Y \cup \bigcup_{U \in \B}C_U.$$
Clearly $X$ is metrizable, because $X \sub D^\w$, and $|X| = |Y|+|\B| \cdot \continuum = \continuum^{+(\w+1)}$.

From our construction of $X$, it is obvious that every nonempty open subset of $X$ contains a copy of $2^\w$. By Lemma~\ref{lem:setup}, $X$ is $2^\w$-coverable.

Yet we claim that there is no partition of $X$ into copies of $2^\w$. Indeed, suppose $\mathcal Q$ is any packing of Cantor spaces into $X$. Every $K \in \mathcal Q$ must contain a point of $\bigcup_{U \in \B}C_U$, because otherwise we would have $K \sub Y$, contradicting to our choice of $Y$. But $\card{\bigcup_{U \in \B}C_U} = |\B| \cdot |2^\w| = \continuum^{+\w} \cdot \continuum = \continuum^{+\w}$, and therefore $|\mathcal Q| \leq \continuum^{+\w}$. But then $\card{\bigcup \mathcal Q} = |\mathcal Q| \cdot \continuum \leq \continuum^{+\w} < |X|$. Therefore $\mathcal Q$ is not a partition of $X$.
\end{proof}

\section{Completely metrizable spaces}

In this section we prove a completely metrizable space is $2^\w$-partitionable if and only if it is $2^\w$-coverable. 
In fact, we prove something a little more general: a metrizable, absolute Borel space is $2^\w$-partitionable if and only if it is $2^\w$-coverable. For completely metrizable spaces, this is also equivalent to having no isolated points.

Recall that a space $X$ is called an \emph{absolute Borel} space if it is homeomorphic to a Borel subspace of a completely metrizable space. Such spaces were studied extensively by Stone in \cite{STONE}.

Recall that the \emph{weight} of a topological space $X$, denoted $\wt(X)$, is the smallest possible cardinality of a basis for $X$, i.e.,
$$\wt(X) = \min\set{|\B|}{\B \text{ is a basis for }X}.$$
Note that $\wt(X) \leq c(X)$ for any $X$. For metrizable spaces, we can say more:

\begin{lemma}\label{lem:wt}
If $X$ is metrizable, then $\wt(X) = c(X)$, and there is a cellular family in $X$ with size $\wt(X)$.
\end{lemma}
\begin{proof}
The first assertion is proved in \cite[Theorem 4.1.15]{Engelking}.
The second assertion is proved by Erd\H{o}s and Tarski in \cite{Erdos&Tarski} (as mentioned in the previous section), and can also be found in \cite[Exercise 4.1.H]{Engelking}.
\end{proof}

%

\begin{lemma}\label{lem:bigly}
Let $X$ be an absolute Borel space. If $X$ is $2^\w$-coverable, then there is a packing of Cantor spaces into $X$ with size $|X|$.
\end{lemma}
\begin{proof}
Let $\k = \wt(X)$. 
By a theorem of Stone (Theorem 10 in subsection 5.4 of \cite{STONE}), there is a continuous bijection $h: Y \to X$, where $Y$ is the disjoint sum of $\leq\!\k$ generalized Baire spaces. (A \emph{generalized Baire space} means a space of the form $D^\w$, where $D$ is discrete. We do not exclude the possibility that $|D| = 1$, in which case $D^\w$ consists of a single point.)
Let $A$ denote the set of all isolated points in $Y$ (i.e., all the degenerate generalized Baire spaces in our disjoint sum), and let $B = Y \setminus A$. 

\vspace{2mm}

\noindent \emph{Case 1:} Suppose $|X| = \k = \wt(X)$.
By Lemma~\ref{lem:wt}, there is a cellular family $\mathcal S$ in $X$ with $\card{\mathcal S} = \k$. Because we are assuming $X$ is $2^\w$-coverable, Lemma~\ref{lem:setup} implies every $U \in \mathcal S$ contains some $K_U$ homeomorphic to $2^\w$. Then $\set{K_U}{U \in \mathcal S}$ is the required packing.

\vspace{2mm}

\noindent \emph{Case 2:} Suppose $|X| > \k$. Note that $|X| = |A|+|B|$, as $X = h[A] \cup h[B]$ and $h$ is a bijection. But $|A| \leq \k$, because $Y$ is the union of $\leq\!\k$ generalized Baire spaces, so in this case we have $|X| = |B|$.

We claim that $B$ can be partitioned into $|B|$ copies of $2^\w$. Because $B$ is a disjoint sum of non-degenerate generalized Baire spaces, i.e. spaces of the form $D^\w$ with $D$ discrete and $|D| \geq 2$, it suffices to show that the space $D^\w$ can be partitioned into $|D^\w|$ copies of $2^\w$ whenever $|D| \geq 2$. But this is easy: if $D$ is finite then $D \homeo 2^\w$ and the result follows from Lemma~\ref{lem:break-up}, and if $D$ is infinite then $D^\w \homeo (\{0,1\} \times D)^\w \homeo 2^\w \times D^\w$.

Let $\P$ be a partition of $B$ into $|B|$ copies of $2^\w$. Because $h$ is a continuous bijection onto a Hausdorff space, $h$ restricts to a homeomorphism on any compact subset of $Y$. In particular, $h[C] \homeo 2^\w$ for every $C \in \P$. Hence $\set{h[C]}{C \in \P}$ is a packing of $|B| = |X|$ Cantor spaces into $X$. 
\end{proof}

\begin{theorem}\label{thm:absBorel}
An absolute Borel space is $2^\w$-partitionable if and only if it is $2^\w$-coverable.
\end{theorem}
\begin{proof}
The ``only if'' part of the theorem holds in general, for any space. For the ``if'' part,
let $X$ be a nonempty absolute Borel space (the theorem is trivially true for $X = \0$).
If $U$ is a nonempty open subset of $X$, then $U$ is also an absolute Borel space. Furthermore, $U$ is $2^\w$-coverable: this follows from Lemma~\ref{lem:setup}, since it is clear that statement $(3)$ in Lemma~\ref{lem:setup} is preserved by passing to open subsets. 
By Lemma~\ref{lem:bigly}, there is a packing of $|U|$ Cantor spaces into $U$.
As $U$ was arbitrary, Theorem~\ref{thm:TFAE} now implies that $X$ is $2^\w$-partitionable.
\end{proof}

%
\begin{theorem}\label{thm:Polish}
A completely metrizable space
is $2^\w$-partitionable if and only if 
it is $2^\w$-coverable if and only if 
it has no isolated points.
\end{theorem}
\begin{proof}
The first ``if and only if'' holds by Theorem~\ref{thm:absBorel}, as every completely metrizable space is absolute Borel, and it is clear that any $2^\w$-coverable space has no isolated points.
So we must show that if a completely metrizable space has no isolated points, then it is $2^\w$-coverable. 
Suppose $X$ is such a space, and let $U$ be a nonempty open subset of $X$. 
Then $U$ is itself completely metrizable, and has no isolated points.
It follows (e.g., via \cite[Lemma 5.3]{STONE}) that $U$ contains a copy of the Cantor space.
As $U$ was arbitrary, Lemma~\ref{lem:setup} now implies $X$ is $2^\w$-coverable.
\end{proof}

\section{A few further things}


\begin{proposition}\label{prop:obvious}
Let $Y$ and $Z$ be topological spaces. If $Y$ is $Z$-coverable, then any $Y$-coverable space is also $Z$-coverable. If $Y$ is $Z$-coverable and $Z$ is $Y$-coverable, then ``$Y$-coverable'' and ``$Z$-coverable'' are equivalent properties.
Similarly, if $Y$ is $Z$-partitionable, then any $Y$-partitionable space is also $Z$-partitionable. If $Y$ is $Z$-partitionable and $Z$ is $Y$-partitionable, then ``$Y$-partitionable'' and ``$Z$-partitionable'' are equivalent properties.
\end{proposition}
\begin{proof}
Suppose $\C$ is a covering of some space $X$ into copies of $Y$. If $Y$ is $Z$-coverable, then every member of $\P$ may be covered by copies of $Z$, and the union of all these coverings is a covering of $X$ with copies of $Z$. This proves the first assertion about coverings, and the second assertion follows immediately from the first.
The two assertions about partitions are proved in exactly the same way.
\end{proof}

\begin{theorem}
Let $Y$ be a zero-dimensional Polish space without isolated points. Then $2^\w$ is $Y$-partitionable and $Y$ is $2^\w$-partitionable. Consequently:
\begin{enumerate}
\item If $X$ is first countable and $|X| \leq \continuum$, then $X$ is $Y$-partitionable if and only if $X$ is $Y$-coverable.
\item If $X$ is metrizable and $|X| \leq \continuum^{+\w}$, then $X$ is $Y$-partitionable if and only if $X$ is $Y$-coverable.
\item If $X$ is completely metrizable, then $X$ is $Y$-partitionable if and only if $X$ is $Y$-coverable if and only if $X$ has no isolated points.
\end{enumerate}
\end{theorem}
\begin{proof}
It follows from the results of the previous sections that $\w^\w$ is $2^\w$-partitionable. Inversely, $2^\w$ is $\w^\w$-partitionable. This is proved as Corollary 1.3 in \cite{Bankston&McGovern}; alternatively, one may check that the proof of Theorem~\ref{thm:gen} goes through verbatim when $2^\w$ is replaced by $\w^\w$, thus showing that every first countable space of size $\leq \! \continuum$ is $\w^\w$-partitionable if and only if it $\w^\w$-coverable (and it is easy to show $2^\w$ is $\w^\w$-coverable). This proves the first assertion of the theorem in the special case $Y = \w^\w$.

More generally, suppose $Y$ is any zero-dimensional Polish space without isolated points. Then $Y$ is $2^\w$-partitionable (by the results of the previous sections), $2^\w$ is $\w^\w$-partitionable (by the previous paragraph), and $\w^\w$ is $Y$ partitionable (because $\w^\w \homeo Y \times \w^\w$ by \cite[Theorem 7.7]{Kechris}). Via a few applications of Proposition~\ref{prop:obvious}, this shows $2^\w$ is $Y$-partitionable and $Y$ is $2^\w$-partitionable.
This also implies $2^\w$ is $Y$-coverable and $Y$ is $2^\w$-coverable.

Statements $(1)$, $(2)$, and $(3)$ follow from Theorems \ref{thm:gen}, \ref{thm:c+omega}, and \ref{thm:Polish}, respectively, together with the previous paragraph and Proposition~\ref{prop:obvious}.
\end{proof}

We note that ``without isolated points'' cannot be removed from the theorem above. For example, let $Y$ denote the convergent sequence $\w+1$. Then any (Hausdorff) space containing a copy of $Y$ is $Y$-coverable. (\emph{Proof:} Let $X$ be a Hausdorff space, and fix $X_0 \sub X$ with $X_0 \homeo \w+1$. Then $\set{X_0 \cup \{x\}}{x \in X \setminus X_0}$ is a covering of $X$ with copies of $\w+1$.)
The space $X = \w+\w$ (i.e., the union of $\w+1$ with countably many isolated points) is $(\w+1)$-coverable but not $(\w+1)$-partitionable, and serves as a counterexample to $(1)$, $(2)$, and $(3)$ above for $Y = \w+1$.

Similarly, if we take $Y$ to be the union of $2^\w$ with a single isolated point, and $X$ to be the union of $2^\w$ with $\continuum^+$ isolated points, $X$ is $Y$-coverable and not $Y$-partitionable. This provides another counterexample to $(2)$ and $(3)$ above, this time for $Y$ containing only a single isolated point.

On the other hand, for any zero-dimensional Polish space $Y$ (even those with isolated points), every $2^\w$-partitionable space is $Y$-partitionable. This follows from Proposition 5.1 and Theorem 5.2 above, because $\w^\w \homeo Y \times \w^\w$ is $Y$-partitionable.

\subsection*{Acknowledgement}

I wish to thank the anonymous referee of this paper, who pointed me to the work of Stone in \cite{STONE}, enabling me to strengthen the results in Section 4 (originally stated just for completely metrizable spaces) while also making the proofs shorter.



\end{document}